\documentclass[12pt]{amsart}
\usepackage{amssymb}
\usepackage{colordvi}
\usepackage{graphicx}
\usepackage{vmargin}
\usepackage{pgf,tikz}
\usetikzlibrary{arrows}
\setpapersize{USletter}
\setmargrb{1in}{1in}{1in}{1in} 
\newtheorem{theorem}{Theorem}[section]

\newtheorem{conjecture}[theorem]{Conjecture}

\newtheorem{defn}[theorem]{Definition}
\newtheorem{ndefn}[theorem]{Na\"{i}ve Definition}

\theoremstyle{definition}

\newtheorem{example}[theorem]{Example}
\newtheorem{remark}[theorem]{Remark}

\newcommand{\R}{\ensuremath{\mathbb{R}}}
\newcommand{\C}{\ensuremath{\mathbb{C}}}
\newcommand{\Z}{\ensuremath{\mathbb{Z}}}
\newcommand{\Q}{\mathbb{Q}}

\newcommand{\ord}{\text{ord}}

\newcommand{\trop}{\text{trop}}
\newcommand{\Trop}{\text{Trop}}
\newcommand{\val}{\text{val}}

\newcommand\be{\begin{equation}}
\newcommand\ee{\end{equation}}
\newcommand\bea{\begin{eqnarray}}
\newcommand\eea{\end{eqnarray}}
\newcommand\bi{\begin{itemize}}
\newcommand\ei{\end{itemize}}
\newcommand\ben{\begin{enumerate}}
\newcommand\een{\end{enumerate}}
\newcommand\ba{\begin{array}}
\newcommand\ea{\end{array}}
\newcommand\bpf{\begin{proof}}
\newcommand\epf{\end{proof}}
\newcommand\bex{\begin{exercise}}
\newcommand\eex{\end{exercise}}

\title{Tropical Images of Intersection Points}

\author{Ralph Morrison}
\address{Ralph Morrison\\
Department of Mathematics\\
      University of California\\
      Berkeley, California 94720,
      USA}
\email{ralphmorrison@berkeley.edu}
\begin{document}

\maketitle

\begin{abstract} A key issue in tropical geometry is the lifting of intersection points to
a non-Archimedean field. Here, we ask: ÒWhere can classical intersection
points of planar curves tropicalize to?Ó An answer should have two parts:
first, identifying constraints on the images of classical intersections,
and, second, showing that all tropical configurations satisfying these
constraints can be achieved. This paper provides the first part: images of
intersection points must be linearly equivalent to the stable tropical
intersection by a suitable rational function. Several examples provide
evidence for the conjecture that our constraints may suffice for part two.
\end{abstract}

\section{Introduction}\label{section:intro}

Let $K$ be an algebraically closed non-Archimedean field with a nontrivial valuation $\text{val}:K^*\rightarrow\R$. The examples throughout this paper will use $K=\C\{\!\{t\}\!\}$, the field of Puiseux series over the complex numbers with indeterminate $t$.  This is the algebraic closure of the field of Laurent series over $\C$, and can be defined as
$$\C\{\!\{t\}\!\}=\left\{\sum_{i=k}^\infty a_i t^{i/n}\,:\, a_i\in \C, n,k\in\Z, n>0 \right\}, $$
with $\val\left(\sum_{i=k}^\infty a_i t^{i/n}\right)=k/n$ if $a_k\neq 0$.  In particular, $\val(t)=1$.

The tropicalization map $\trop:(K^*)^n\rightarrow\R^n$ sends points in the $n$-dimensional torus $(K^*)^n$ into Euclidean space under coordinate-wise valuation:
$$\trop:(a_1,\ldots,a_n)\rightarrow(\val(a_1),\ldots,\val(a_n)).$$
In tropical geometry, we consider the tropicalization map on a variety $X\subset (K^*)^n$.  Since the value group is dense in $\R$, we take the Euclidean closure of $\trop(X)$ in $\R^n$, and call this the \emph{tropicalization of $X$}, denoted $\Trop(X)$.  The tropicalization of a variety is a piece-wise linear subset of $\R^n$, and has the structure of a balanced weighted polyhedral complex.  In the case where $X$ is a hypersurface, the combinatorics of the tropicalization can be found from a subdivision of the Newton polytope of $X$.  For more background on tropical geometry, see \cite{Gu} and \cite{MS}.


Consider two curves $X,Y\subset (K^*)^2$ intersecting in a finite number of points. We are interested in the image of the intersection points under tropicalization; that is, in $\Trop(X\cap Y)$ inside of $\Trop(X)\cap\Trop(Y)\subset\R^2$.  It was shown in \cite[Theorem 1.1]{OP} that if $\Trop(X)\cap\Trop(Y)$ is zero dimensional in a neighborhood of a point in the intersection, then that point is in $\Trop(X\cap Y)$.  More generally, they showed this for varieties $X$ and $Y$ under the assumption that $\Trop(X)\cap\Trop(Y)$ has codimension $\text{codim }X+\text{codim }Y$ in a neighborhood of the point.  It follows that if $\Trop(X)\cap\Trop(Y)$ is a finite set, then $\Trop(X\cap Y)=\Trop(X)\cap\Trop(Y)$.

It is possible for $\Trop(X)\cap\Trop(Y)$ to have higher dimensional components, namely finite unions of line segments and rays.  It was shown in \cite{OR} that if $\Trop(X)\cap\Trop(Y)$ is bounded, then each connected component of $\Trop(X)\cap\Trop(Y)$ has the ``right'' number of images of points in $X\cap Y$, counted with multiplicity.  In this context, the ``right'' number is the number of points in the stable tropical intersection of that connected component; the stable tropical intersection is $\lim_{\varepsilon\rightarrow 0}(\Trop(X)+\varepsilon\cdot v)\cap \Trop(Y)$, where $v$ is a generic vector and $\varepsilon$ is a real number \cite[\S4]{OR}.  They further showed that the theorem holds for components of $\Trop(X)\cap\Trop(Y)$ that are unbounded, after a suitable compactification.

We offer the following example to illustrate this higher dimensional component phenomenon. This will motivate the following question:  as we vary $X$ and $Y$ over curves with the same tropicalizations, what are the possibilities for the varying set $\Trop(X\cap Y)$ inside of  the fixed set $\Trop(X)\cap\Trop(Y)$?

\begin{example}\label{motivating_example}

 Let $K=\C\{\!\{t\}\!\}$ and let $f,g\in K[x,y]$ be $f(x,y)=c_1+c_2x+c_3y $ and $g(x,y)=c_4x+c_5xy+tc_6y$, where $c_i\in K$ and  $\val(c_i)=0$ for all $i$.  Let $X,Y\subset (K^*)^2$ be the curves defined by $f$ and $g$, respectively.
 
  \begin{figure}[hbt]
\begin{center}
\includegraphics[
height=1.4in
]
{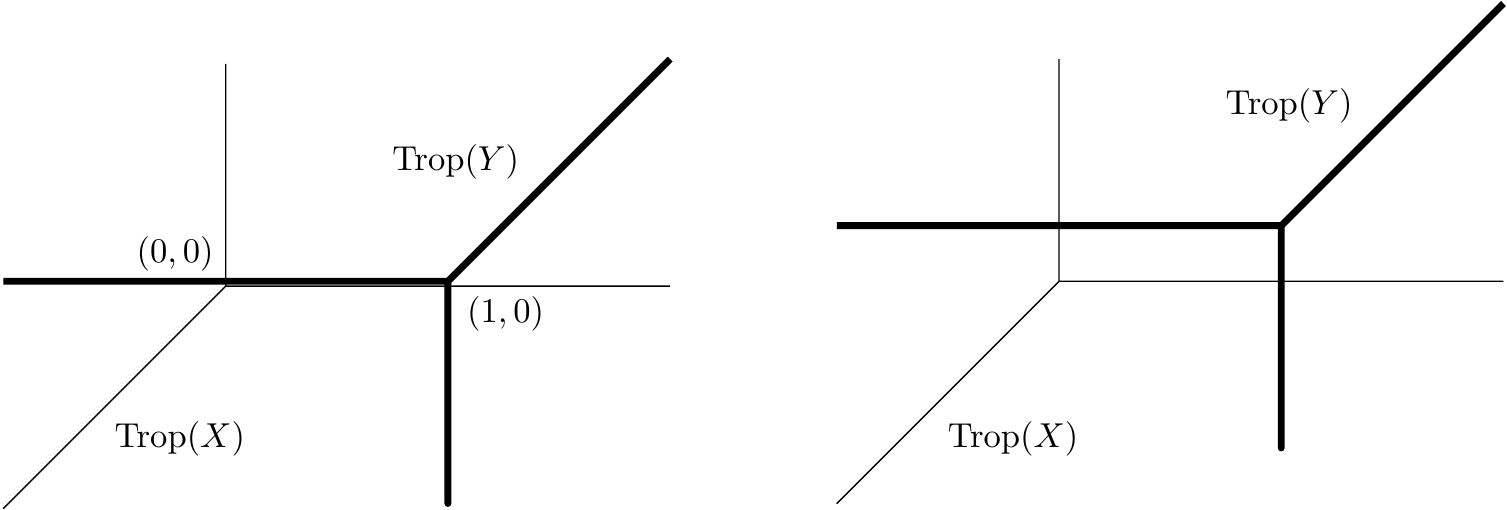}
\end{center}
\caption{ $\Trop(X)$ and $\Trop(Y)$, before and after a small shift of $\Trop(Y)$.}
\label{figure:line_conic}
\end{figure}
 
Regardless of our choice of $c_i$, $\Trop(X)$ and $\Trop(Y)$ will be as pictured in Figure \ref{figure:line_conic}, with $\Trop(X)$ and $\Trop(Y)$ intersecting in the line segment $L$ from $(0,0)$ to $(1,0)$.   However, $X$ and $Y$ only intersect in two points (or one point with multiplicity $2$).  The natural question is: as we vary the coefficients while keeping valuations (and thus tropicalizations) fixed, what are the possible images of the two intersection points within $L$?

 A reasonable guess is that the intersection points map to the stable tropical intersection $\{(0,0),(1,0)\}$, and indeed this does happen for a generic choice of coefficients.  However, as shown in Example \ref{motivating_example_detailed}, one can choose coefficients such that the intersection points map to any pair of points in $L$ of the form $(r,0)$ and $(1-r,0)$, where $0\leq r\leq \frac{1}{2}$.  These possible configurations are illustrated in Figure \ref{figure:line_conic_configurations}.  
 
  \begin{figure}[hbt]
\begin{center}
\includegraphics[
width=4.5in
]
{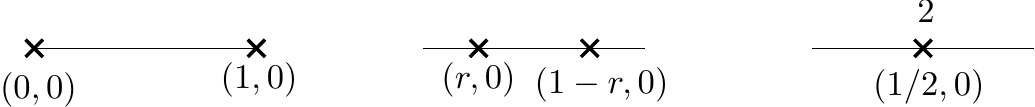}
\end{center}
\caption{Possible images of $X\cap Y$ in $\Trop(X)\cap\Trop(Y)$.}
\label{figure:line_conic_configurations}
\end{figure}
\end{example}

The main result of this paper is that the points $\Trop(X\cap Y)$ inside of $\Trop(X)\cap\Trop(Y)$ must be linearly equivalent to the stable tropical intersection via particular \emph{tropical rational functions}, defined in Section \ref{section:background}. To distinguish tropical rational functions from classical rational functions, they will be written as $f^{\trop}$, $g^{\trop}$, or $h^{\trop}$ instead of $f$, $g$, or $h$. See \cite{BN} and  \cite{GK} for more background.  In all the examples discussed in Section \ref{section:examples}, essentially every such configuration is achievable.  Conjecture \ref{main_conjecture} expresses our hope that this always holds.

\begin{theorem}\label{main_theorem}  Let $X,Y\subset (K^*)^2$ where $X\cap Y$ is equal to the multiset $\{p_1,\ldots,p_n\}$ and where $\Trop(X)$ is smooth.  Let $E$ be the stable intersection divisor of $\Trop(X)$ and $\Trop(Y)$, and let $D$ be
 $$D=\sum_i\trop(p_i).$$
 Then there exists a tropical rational function $h^{\trop}$ on $\Trop(X)$ such that $(h^{\trop})=D-E$ and $\text{supp}(h^{\trop})\subset \Trop(X)\cap\Trop(Y)$.
 \end{theorem}
 
 We will present two proofs of this theorem.  In Sections \ref{section:background} and \ref{section:main_result} we approach the question from the perspective of Berkovich theory, which in the smooth case allows us to tropicalize rational functions on classical curves.  In Section \ref{section:modifications} we present an alternate argument using tropical modifications, which allows us to drop the smoothness assumption.
 
\begin{example}\label{motivating_example_extended}
Let $X$ and $Y$ be as in Example \ref{motivating_example}. We will consider tropical rational functions on $\Trop(X)\cap \Trop(Y)$ such that
\begin{itemize}
\item[(i)] the stable intersection points are the poles (possibly canceling with zeros), and
\item[(ii)]  the tropical rational function takes on the same value at every boundary point of $\Trop(X)\cap\Trop(Y)$.
\end{itemize}
If we insist that the ``same value'' in condition (ii) is $0$, we may extend these tropical rational functions to all of $\Trop(X)$ by setting them equal to $0$ on $\Trop(X)\setminus \Trop(Y)$.  This yields tropical rational functions on $\Trop(X)$ with  $\text{supp}(h^{\trop})\subset \Trop(X)\cap\Trop(Y)$, as in Theorem \ref{main_theorem}.  Instances of the types of such tropical rational functions on $L=\Trop(X)\cap\Trop(Y)$ from our example are illustrated in Figure \ref{figure:line_conic_rational_functions}.  

  \begin{figure}[hbt]
\begin{center}
\includegraphics[
width=4.5in
]
{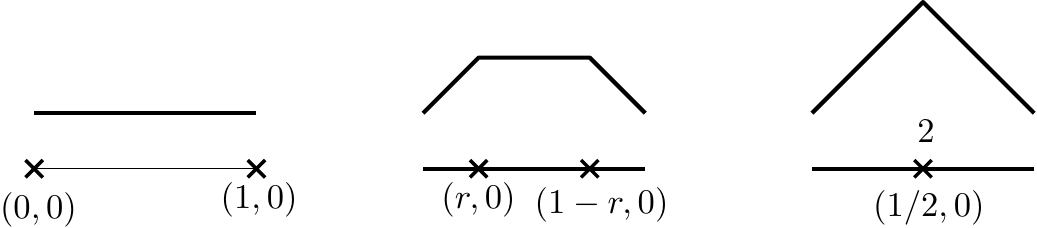}
\end{center}
\caption{Graphs of tropical rational functions  on $\Trop(X)\cap\Trop(Y)$.}
\label{figure:line_conic_rational_functions}
\end{figure}

As asserted by Theorem \ref{main_theorem}, all possible image intersection sets in $\Trop(X)\cap\Trop(Y)$ arise as the zero set of such a tropical rational function.  Equivalently, the stable intersection divisor and the image of intersection divisor are linearly equivalent via one of these functions.  
\end{example}

\begin{remark}It is not quite the case that the zero set of every such tropical rational function (from Example \ref{motivating_example_extended}) is attainable as the image of the intersections of $X$ and $Y$ (with changed coefficients).  For instance, such a tropical rational function could have zeros at $(\frac{\sqrt{2}}{2},0)$ and $(1-\frac{\sqrt{2}}{2},0)$, which cannot be the images of \emph{any} points on $X$ and $Y$ since they have irrational coordinates.  However, if we insist that the tropical rational functions have zeros at points with rational coefficients (since $\Q=\val(K^*)$), all zero sets can be achieved as the images of intersections.  This is the content of Conjecture \ref{main_conjecture}.
\end{remark}

\subsection*{Acknowledgements}  The author would like to thank Matt Baker and Bernd Sturmfels for introducing him to these questions in tropical geometry.  The author would also like to thank Sarah Brodsky, Melody Chan, Nikita Kalinin, Kristin Shaw, and Josephine Yu for helpful conversations and insights.  The author was supported by the NSF through grant DMS-0968882 and a graduate research fellowship, and by the Max Planck Institute for Mathematics in Bonn.

\section{Tropicalizations of Rational Functions}\label{section:background}

In this section we present background information on tropical rational function theory, and use some Berkovich theory to define the tropicalization of a rational function.  For the theory of tropical rational functions, we consider abstract tropical curves $\Gamma$, which are weighted metric graphs with finitely many edges and vertices, where the edges have possibly infinite lengths.  See \cite{BPR} for background on Berkovich spaces, and \cite{Mi} for more background on tropical rational functions.

Tropical rational functions on tropical curves are analogous to classical rational functions on classical curves. A \emph{divisor} on a tropical curve $\Gamma$ is a finite formal sum of points in $\Gamma$ with coefficients in $\Z$.  If $D=\sum_ia_iP_i$, the \emph{degree} of $D$ is $\deg D:=\sum_ia_i$.  The \emph{support} of $D$ is the set of all points $P_i$ with $a_i\neq 0$, and $D$ is called \emph{effective} if all $a_i$'s are nonnegative.  

\begin{defn}\rm{  A \emph{rational function} on a tropical curve $\Gamma$ is a continuous function $f^{\trop}:\Gamma\rightarrow \R\cup\{\pm\infty\}$ such that the restriction of $f^{\trop}$ to any edge of $\Gamma$ is a piecewise linear function with integer slopes and only finitely many pieces.  This means that $f^{\trop}$ can only take on the values of $\pm\infty$ at the unbounded ends of $\Gamma$.  The \emph{associated divisor} of $f^{\trop}$ is $(f^{\trop})=\sum_{P\in\Gamma}\ord_P(f^{\trop})\cdot P$}, where $\ord_P(f^{\trop})$ is minus the sum of the outgoing slopes of $f$ at a point $P$.   If $D$ and $E$ are divisors such that $D-E=(f^{\trop})$ for some tropical rational function $f$, we say that $D$ and $E$ are \emph{linearly equivalent}.
\end{defn}

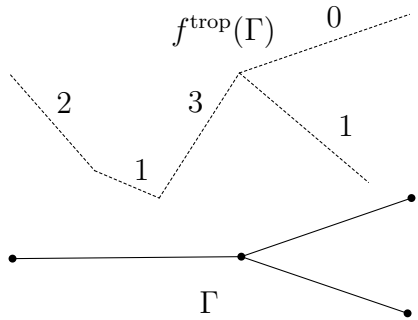
\begin{figure}[hbt]
\begin{tikzpicture}[line cap=round,line join=round,>=triangle 45,x=1.3cm,y=1.3cm]
\clip(3.64,5.38) rectangle (8.89,9.57);
\draw (4,6.48)-- (6.34,6.5);
\draw (6.34,6.5)-- (8.04,5.92);
\draw (6.34,6.5)-- (8.08,7.1);
\draw [dash pattern=on 1pt off 1pt] (3.98,8.36)-- (4.84,7.38);
\draw [dash pattern=on 1pt off 1pt] (4.84,7.38)-- (5.5,7.1);
\draw [dash pattern=on 1pt off 1pt] (5.5,7.1)-- (6.32,8.38);
\draw [dash pattern=on 1pt off 1pt] (6.32,8.38)-- (7.64,7.26);
\draw [dash pattern=on 1pt off 1pt] (6.32,8.38)-- (8.06,8.98);
\draw (4.34,8.29) node[anchor=north west] {$ 2 $};
\draw (5.13,7.6) node[anchor=north west] {$ 1$};
\draw (5.68,8.3) node[anchor=north west] {$ 3 $};
\draw (7.1,9.17) node[anchor=north west] {$ 0 $};
\draw (7.22,8.06) node[anchor=north west] {$ 1 $};
\draw (5.8,6.26) node[anchor=north west] {$ \Gamma $};
\draw (5.51,9.08) node[anchor=north west] {$ f^{\trop}(\Gamma) $};
\begin{scriptsize}
\fill [color=black] (4,6.48) circle (1.5pt);
\fill [color=black] (6.34,6.5) circle (1.5pt);
\fill [color=black] (8.04,5.92) circle (1.5pt);
\fill [color=black] (8.08,7.1) circle (1.5pt);
\end{scriptsize}
\end{tikzpicture}
\caption{The graph of a rational function $f^{\trop}$ on an abstract tropical curve $\Gamma$.}
\label{rationalfunction}
\end{figure}

As an example, consider Figure \ref{rationalfunction}.  Here $\Gamma$ consists of four vertices and three edges arranged in a Y-shape, and the image of $\Gamma$ under a rational function $f$ is illustrated lying above it.  The leftmost vertex is a zero of order $2$, since there is an outgoing slope of $-2$ and no other outgoing slopes.  The next kink in the graph is a pole of order $1$, since  the outgoing slopes are $2$ and $-1$ and $2+(-1)=1$.  Moving along in this direction we have a pole of order $4$, a zero of order $4$, at one endpoint a pole of order $1$, and at the other endpoint no zeros or poles. Note that, counting multiplicity, there are six zeros and six poles.  The numbers agree, as in the classical case.

Since we can tropicalize a curve to obtain a tropical curve, we would like to tropicalize a rational function on a curve and obtain a tropical rational function on a tropical curve.  A na\"{i}ve definition of ``tropicalizing a rational function'' would be as follows.

\begin{ndefn}Let $h$ be a rational function on a curve $X$.  Define the \emph{tropicalization of $h$}, denoted $\trop(h)$, as follows.  For every point $w$ in the image of $X\setminus\{\text{zeros and poles of $h$}\}$ under tropicalization, lift that point to $p\in X$, and define
$$\trop(h)(w)=\val(h( p)).$$
Extend this function to all of $\Trop(X)$ by continuity.
\end{ndefn}

Unfortunately this is not quite well-defined, because $\val(h( p))$ depends on which lift $p$ of $w$ we choose.  However, as suggested to the author by Matt Baker, this definition can be made rigorous if at least one of the tropicalizations is suitably faithful in a Berkovich sense. Let $h$ be a rational function on $X$, and assume that there is a canonical section $s$ to the map $X^{an}\rightarrow \Trop(X)$, where $X^{an}$ is the analytification of $X$.  For $w\in \Trop(X)$,  define
$$\trop(h)(w)=\log|h|_{s(w)},$$
where $|\cdot|_{s(w)}$ is the seminorm corresponding to the point $s(w)$ in $X^{an}$.  This rational function has the desired properties.

\begin{remark}\label{remark_suitably_faithful}  In \cite{BPR} one can find conditions to guarantee that there exists a canonical section $s$ to the map $X^{an}\rightarrow \Trop(X)$.  For instance,  if $\Trop(X)$ is smooth in the sense that it comes from a unimodular triangulation of its Newton polygon, such a section will exist.
\end{remark}

\section{Main Result and a Conjecture}\label{section:main_result}

We are ready to prove Theorem \ref{main_theorem}.

\begin{proof}[Proof of Theorem \ref{main_theorem}]  Let $f$ and $g$ be the defining equations of $X$ and $Y$, respectively. Let $g'\in K[x,y]$ have the same tropical polynomial as $g$, and let $Y'$ be the curve defined by $g'$.  We have that $\Trop(Y)=\Trop(Y')$, and for generic $g'$ we have that $\Trop(X\cap Y')$ is the stable tropical intersection of $\Trop(X)$ and $\Trop(Y)$.

  Recall that $p_1,\ldots,p_n$ denote the intersection points of $X$ and $Y$, possibly with repeats. Let $p'_1,\ldots,p'_m$ denote the intersection points of $X$ and $Y'$, with duplicates in the case of multiplicity.  Note that $m$ and $n$ will be equal unless $X$ and $Y$ have intersection points outside of $(K^*)^2$; this is discussed in Remark \ref{intersections_at_infinity}.

Consider the rational function $h=\frac{g}{g'}$ on $X$, which has zeros at the intersection points of $X$ and $Y$ and poles at the intersection points of $X$ and $Y'$.  Since $\Trop(X)$ is smooth, by Remark \ref{remark_suitably_faithful} we may tropicalize $h$.  This gives a tropical rational function $\trop(h)$ on $\Trop(X)$ with divisor
$$(\trop(h))=\trop(p_1)+\ldots+\trop(p_n)-\trop(p'_1)-\ldots-\trop(p'_m)=D-E.$$

We claim that $\trop(h)$ is the desired $h^{trop}$ from the statement of the theorem.  All that remains to show is that $\text{supp}(\trop(h))\subset \Trop(X)\cap\Trop(Y)$.   If $w\in \Trop(X)\setminus \Trop(Y)$, then $|g|_{s(w)}=|g'|_{s(w)}$ because $g$ and $g'$ both have bend locus $\Trop(Y)$, and $w$ is away from $\Trop(Y)$.  This means that $\trop(h)(w)=|h|_{s(w)}=|g|_{s(w)}-|g'|_{s(w)}=0$ on $\Trop(X)\setminus \Trop(Y)$.  This completes the proof.  
 \end{proof}

\begin{remark}  The argument and result will hold even if $\Trop(X)$ is not smooth as long as there exists a section $s$ to $X^{an}\rightarrow \Trop(X)$.
\end{remark}

\begin{remark}  Since we have our result in terms of  linear equivalence, we get as a corollary that the configurations of points differ by a sequence of chip firing moves by \cite{HMY}.
\end{remark}

\begin{remark}\label{intersections_at_infinity} If $\Trop(X)\cap\Trop(Y)$ is unbounded (for instance, if $\Trop(X)=\Trop(Y)$), then it is possible to have zeros of the rational function ``at infinity.''  This is OK, and can be made sense of using a compactifying fan as in \cite[\S3]{OR}.  See Example \ref{doubleline} for an instance of this phenomenon.
\end{remark}

Our theorem has placed a constraint on the configurations of intersection points mapping into tropicalizations.  The following conjecture posits that essentially all these configurations are attainable.

\begin{conjecture}\label{main_conjecture}  Assume we are given $\Trop(X)$ and $\Trop(Y)$ and a tropical rational function $h^{\trop}$ on $\Trop(X)$  with simple poles precisely at the stable tropical intersection points and zeros in some configuration (possibly canceling some of the poles) with coordinates in the value group ($\Q$ for $\C\{\{t\}\}$), such that $\text{supp}(h^{\trop})\subset \Trop(X)\cap\Trop(Y)$.  Then it is possible to find $X$ and $Y$ with the given tropicalizations such that $\trop(p_1),\ldots,\trop(p_n)$ are the zeros of $h^{\trop}$.
\end{conjecture}

\bpf[Proof Strategy]  We will consider the space of all configurations of zeros of rational functions on $\Trop(X)\cap\Trop(Y)$ satisfying the given properties.  This will form a polyhedral complex.  

\bi

\item  First, we will prove that we can achieve the configurations corresponding to the vertices of this complex.

\item  Next, let $E$ be an edge connecting $V$ and $V'$, where the configuration given by $V$ is achieved by $X$ and $Y$ and the configuration given by $V'$ is achieved by $X'$ and $Y'$.  We will prove that we can achieve any configuration along the edge by somehow deforming $(X,Y)$ to $(X',Y')$.  This will show that all points on edges of the complex correspond to achievable configurations.

\item  We will continue this process (vertices give edges, edges give faces, etc.) to show that all points in the complex correspond to achievable configurations.
\ei
\epf
For an illustration of this process, see Example \ref{example_cc} and Figure \ref{mscc}.

\section{Tropical Modifactions}\label{section:modifications}

In this section we outline an alternate proof to Theorem \ref{main_theorem} using tropical modifications.  See \cite[\textsection 4]{BL} for background on this subject.

\begin{proof}[Outline of proof of Theorem \ref{main_theorem} using tropical modifications]
Let $X$, $Y$, $f$, $g$, $g'$, $D$, and $E$ be as in the proof from Section \ref{section:main_result}.  Let $g_{\trop}$  and $g'_{\trop}$ be the tropical polynomials defined by $g$ and $g'$, respectively.

Let $g(X)\subset (K^*)^2\times K$ be the curve that is the closure of $\{(p,g(p )\,|\,p\in X\}$.  Its tropicalization $\Trop(g(X))$ is contained in the tropical hypersurface in $\mathbb{R}^3$ determined by the polynomial $z=g_{\trop}$, and projects onto $\Trop(X)$.  Call this projection $\pi$.  Note that outside of $\Trop(Y)$, $\pi$ is one-to-one, and $\Trop(g(X))$ agrees with $\Trop(g'(X))$.

By \cite[Lemma 4.4]{BL}, the infinite vertical rays in $\pi^{-1}(\Trop(X)\cap\Trop(Y))$ correspond to the intersection points of $X$ and $Y$, and so lie above the support of the divisor $D$ on $\Trop(X)$.  Delete the vertical rays from $\pi^{-1}(\Trop(X)\cap\Trop(Y))$, and decompose the remaining line segments into one or more layer, where each layer gives the graph of a piecewise linear function on $\Trop(X)\cap\Trop(Y)$.  (If deleting the vertical rays makes $\pi$ a bijection, there will be only one layer.)  Call these piecewise linear functions $\ell_1,\ldots,\ell_k$.  The tropical rational function
$$h^{\trop}=\sum_{i=1}^k(\ell_i-g'_{\trop}) $$
has value $0$ outside of $\Trop(X)\cap\Trop(Y)$ because of the agreement of $\Trop(g(X))$ and $\Trop(g'(X))$, and has divisor $D-E$.  
\end{proof}

This argument gives us a slightly stronger version of Theorem \ref{main_theorem}, in that it does not require the assumption of smoothness on $X$.

\section{Evidence for Conjecture \ref{main_conjecture}}\label{section:examples}

In these examples we consider curves $X$ and $Y$ over the field of Puiseux series $\mathbb{C}\{\{t\}\}$.

\begin{example}\label{motivating_example_detailed}  Let $f$ and $g$ be as in Example \ref{motivating_example}.  Treating them as elements of $(K[x])[y]$, their resultant is
$$-c_2c_5x^2 +(c_3c_4 -c_1c_5- tc_2c_6)x - tc_1c_6$$
The two roots of this quadratic polynomial in $x$, which are the $x$-coordinates of the two points in $X\cap Y$, have valuations equal to the slopes of the Newton polygon.  Generically the valuations of the coefficients are $0$, $0$, and $1$, giving slopes $0$ and $1$.  For any rational number $r>0$ we may choose $c_1=1-t^r-t$ and all other $c_i=1$, giving $\val(c_3c_4 -c_1c_5- tc_2c_6)=\val(t^r)=r$.  If $r\leq\frac{1}{2}$ this gives slopes of $r$ and $1-r$, and if $r\geq\frac{1}{2}$ this gives two slopes of $\frac{1}{2}$.  These cases are illustrated in Figure \ref{figure:line_conic_configurations} and correspond to rational functions illustrated in Figure \ref{figure:line_conic_rational_functions}. This means all possible images of intersections allowed by Theorem \ref{main_theorem} with rational coordinates are achievable, so Conjecture \ref{main_conjecture} holds for this example.
\end{example}

\begin{example}\label{example_cc}
Consider conic curves $X$ and $Y$ given by the polynomials $f(x,y)=c_1x+c_2y+c_3xy=0\}$ and $g(x,y)=c_4x+c_5y+c_6xy+t(c_7x^2+c_8y^2+c_9)=0$, where $\val(c_i)=0$ for all $i$. The tropicalizations of $X$ and $Y$ are shown in Figure \ref{cc}, and intersect in three line segments joined at a point.
 
 \begin{figure}[hbt]
\begin{center}
\includegraphics[
height=2.4in
]
{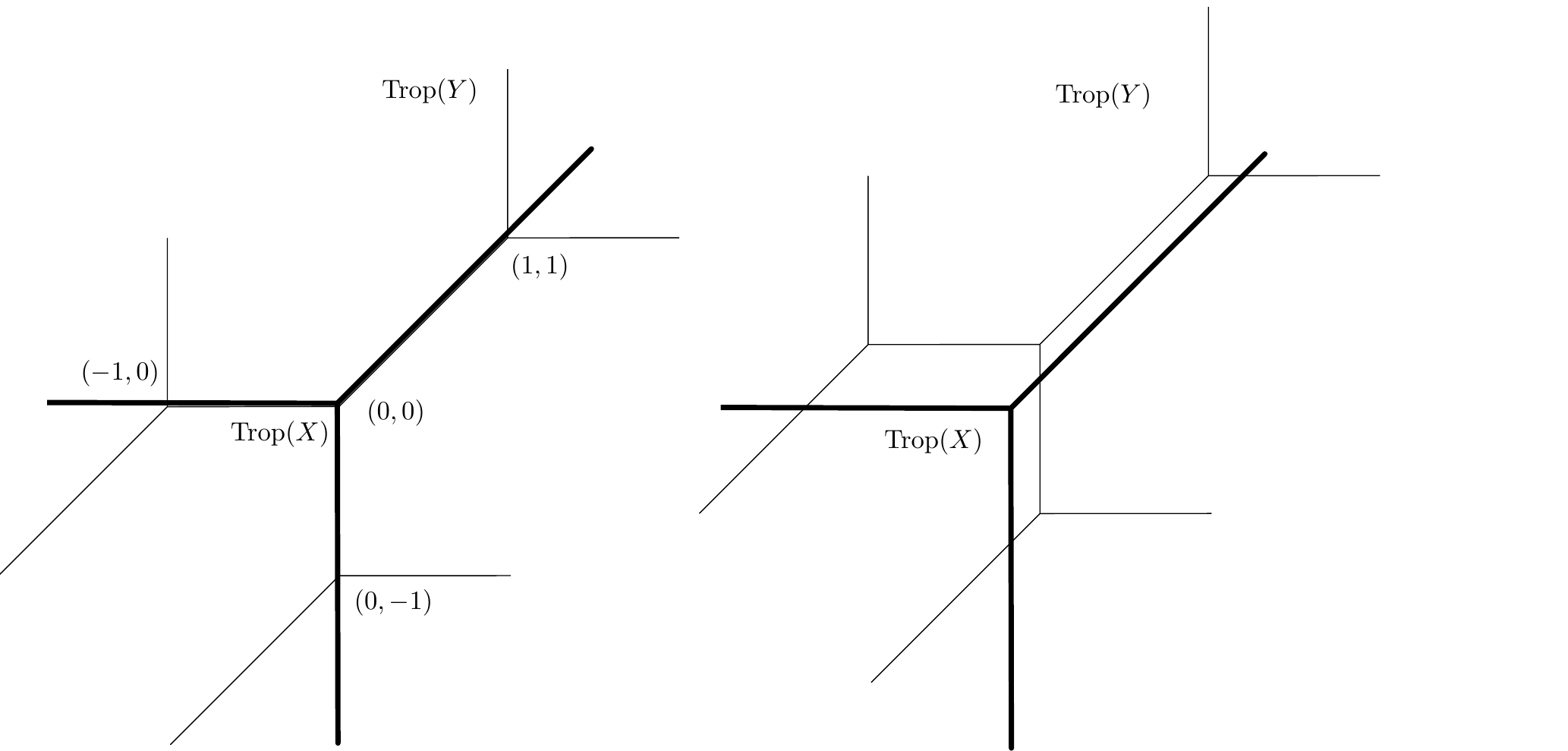}
\end{center}
\caption{$\Trop(X)$ and $\Trop(Y)$, before and after a small shift of $\Trop(Y)$.}
\label{cc}
\end{figure}

The stable tropical intersection consists of four points: $(-1,0)$, $(0,-1)$, $(1,1)$, and $(0,0)$.  The possible images of $\Trop(X\cap Y)$ must be linearly equivalent to these via a rational function equal to $0$ on the three exterior points.  This gives us intersection configurations of three possible types:
\bi
\item[(i)]$\{(-(p-r),0),(0,-p), (p,p),  (-r,0)\}$ where $0\leq r\leq p/2$;
\item[(ii)]  $\{(-p,0),(0,-(p-r)), (p,p),  (0,-r)\}$ where $0\leq r\leq p/2$; and
\item[(iii)] $\{(-p,0),(0,-p), (p-r,p-r),  (r,r)\}$ where $0\leq r\leq p/2$.
\ei
To achieve a type (i) configuration, set $f(x,y)=x+y+xy$ and $g(x,y)=(1+2t^{1-p+r})x+(1+t^{1-p})y+xy+t(x^2+y^2+1)$; if $r>0$, the $2$ can be omitted from the coefficient of $x$ in $g$. The Newton polygons of two polynomials, namely the resultants of $f$ and $g$ with respect to $x$ and with respect to $y$, show that $\Trop(X\cap Y)=\{(-(p-r),0),(0,-p), (p,p),  (-r,0)\}$.  Type (ii) and (iii) are achieved similarly, so Conjecture  \ref{main_conjecture} holds for this example.

For instance, if $f(x,y)=x+y+xy$ and $g(x,y)=(1+t^{1/2})x+(1+t^{1/3})y+xy+t(x^2+y^2+1)$, then $\Trop(X\cap Y)=\{(2/3,2/3),(0,-2/3),(-1/2,0),(-1/6,0)\}$.  The formal sum of these points is linearly equivalent to the stable intersection divisor, as illustrated by the rational function in Figure \ref{trophgraph}.  This is the tropicalization of the rational function $h(x,y)=\frac{(1+t^{1/2})x+(1+t^{1/3})y+xy+t(x^2+y^2+1)}{2x+4y+xy+t(x^2+y^2+1)}$, where $g'(x,y):=2x+4y+xy+t(x^2+y^2+1)$ was chosen so that $\Trop(X)\cap\Trop(V(g'))$ is the stable tropical intersection of $\Trop(X)$ and $\Trop(Y)$.

\begin{figure}[hbt]
\includegraphics[
height=1.5in
]
{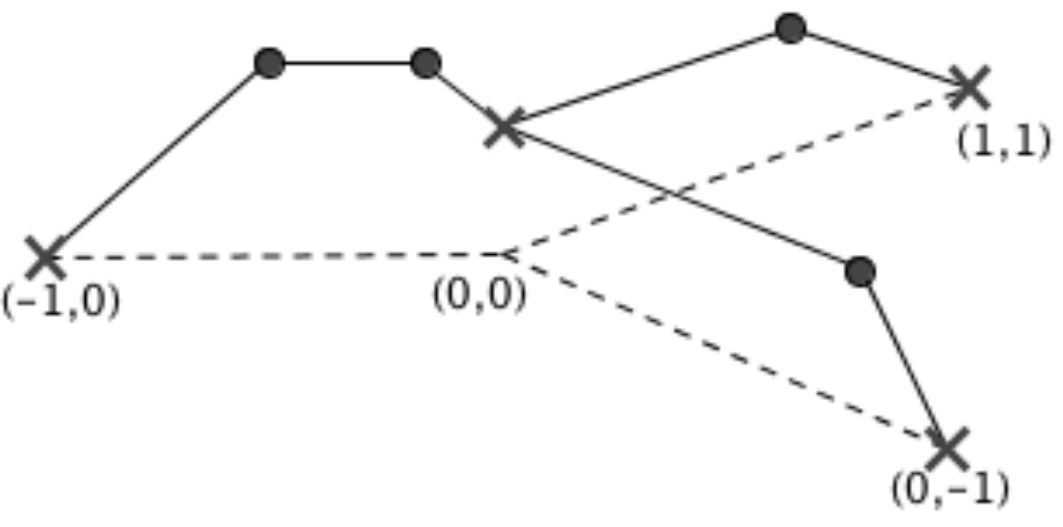}
\caption{The graph of $\trop(h)$ on $\Trop(X)\cap\Trop(Y)$, with zeros at the dots and poles at the x's.}
\label{trophgraph}
\end{figure}

\begin{figure}[hbt]
\begin{center}
\includegraphics[
height=2.1in
]
{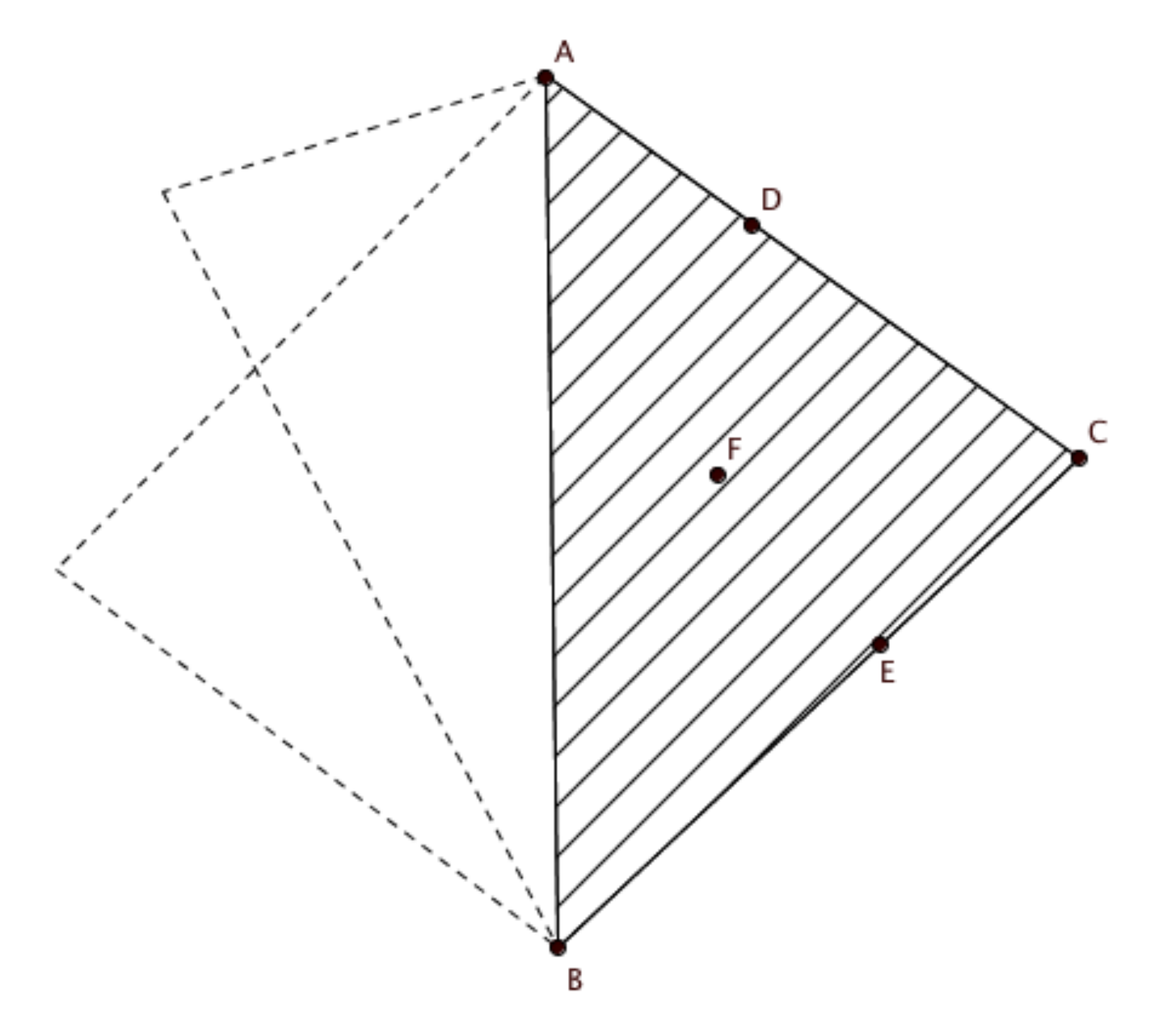}
\includegraphics[
height=2.1in
]
{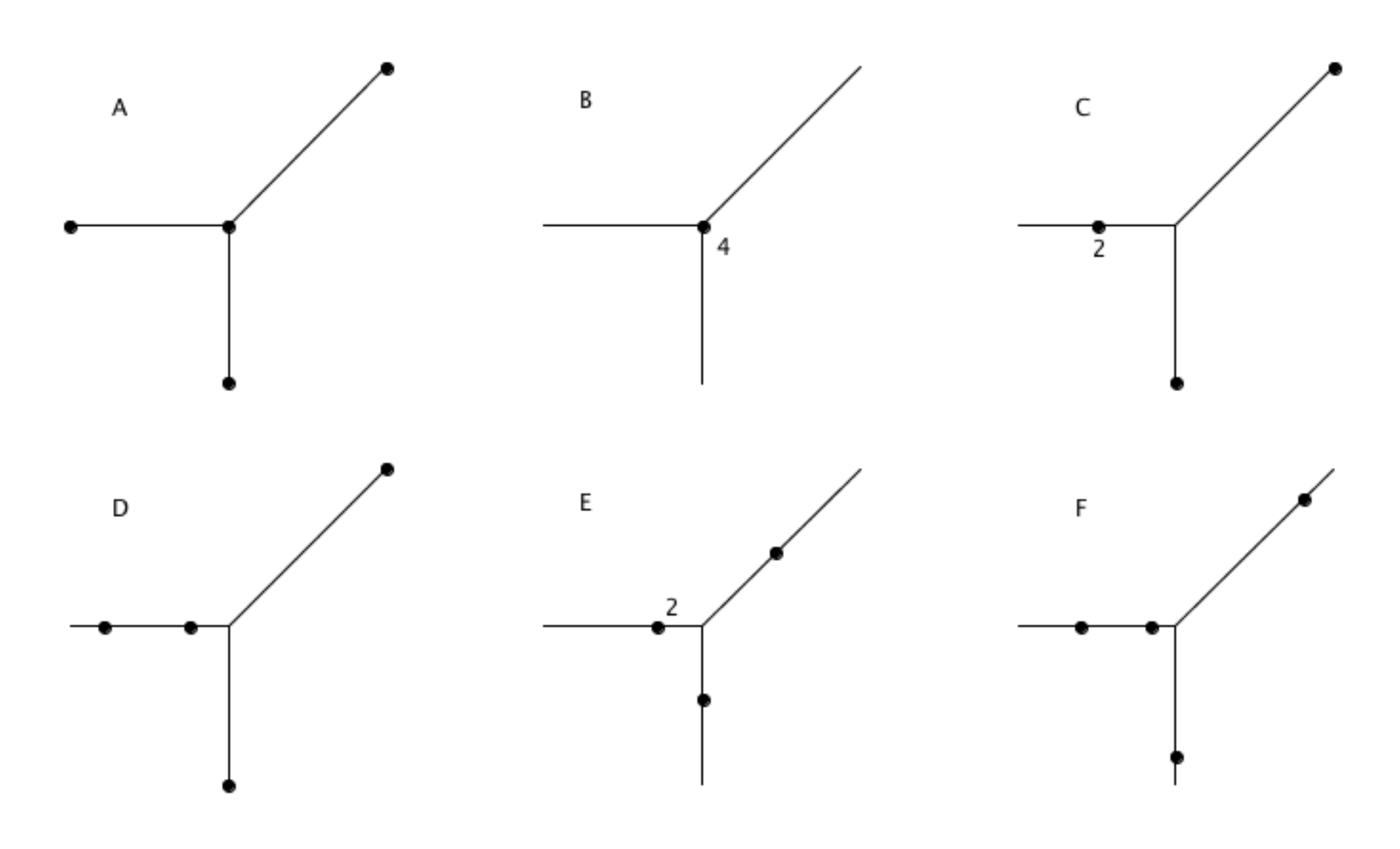}
\end{center}
\caption{The moduli space $M$ of intersection configurations, with six examples.}
\label{mscc}
\end{figure}

We can also consider this example in view of the outlined method of proof for Conjecture \ref{main_conjecture}.  Considering each intersection configuration as a point in $\R^8$ (natural for four points in $\R^2$), we obtain a moduli space $M$ for the possible tropical images of $X\cap Y$. The structure of this space is related to the notion of tropical convexity, as discussed in \cite{Lu}. As illustrated in Figure \ref{mscc},   $M$ consists of three triangles glued along one edge. The hope is that if vertices like $A$ and $C$ can be achieved, then it is possible to slide along the edge and achieve points like $D$.  For instance, if we set
$$f_A(x,y)=f_C(x,y)=f_{{AC},r}=x+y+xy$$
$$g_A=(1+t^0)x+4y+xy+t(x^2+y^2+1)$$
$$g_C=(1+t^{1/2})x+4y+xy+t(x^2+y^2+1)$$
$$g_{{AC},r}=(1+t^{r})x+4y+xy+t(x^2+y^2+1),$$
then $f_A$ and $g_A$ give configuration $A$, $f_C$ and $g_C$ give configuration $C$, and $f_{AC,r}$ and $g_{AC,r}$ give all configurations along the edge $AC$ as $r$ varies from $0$ to $\frac{1}{2}$.

\end{example}

\begin{example}\label{doubleline} Let $X$ and $Y$ be distinct lines defined by $f(x,y)=c_1+c_2x+c_3y$ and $g(x,y)=c_6+c_4x+c_5y$ with $\val(c_i)=0$ for all $i$.  These lines tropicalize to the same tropical line centered at the origin, with  stable tropical intersection equal to the single point $(0,0)$.  Any point on $\Trop(X)=\Trop(X)\cap\Trop(Y)$ is linearly equivalent to $(0,0)$ via a tropical rational function on $X$, so Theorem \ref{main_theorem} puts no restrictions on the image of $p=X\cap Y$ under tropicalization.  In keeping with Conjecture \ref{main_conjecture}, all possibilities can be achieved:

\bi  

\item[(i)] For $\trop( p)=(r,0)$, let $f(x,y)=1+x+y$, $g(x,y)=(1+t^r)+x+y$.

\item[(ii)] For $\trop( p)=(0,r)$, let $f(x,y)=1+x+y$, $g(x,y)=1+(1+t^r)x+y$.

\item[(iii)] For $\trop(p )=(-r,-r)$, let $f(x,y)=1+x+y$, $g(x,y)=1+x+(1+t^r)y$.

\ei

The point $(0,0)$ is also linearly equivalent to points at infinity, as witnessed by rational functions with constant slope $1$ on an entire infinite ray. Mapping $p$ ``to infinity'' means that $X$ and $Y$ cannot intersect in $(K^*)^2$, so we can choose equations for $X$ and $Y$ that give $p$ a coordinate equal to $0$, such as $x+y+1=0$ and $x+2y+1=0$.

\end{example}

\begin{example} Let $X$ and $Y$ be the curves defined by
$$f(x,y)=xy+t(c_1x+c_2y^2+c_3x^2y)$$
$$g(x,y)=xy+t(d_1x+d_2y^2+d_3x^2y) $$
respectively, where $\val(c_i)=\val(d_i)=0$ for all $i$.  This means $\Trop(X)$ and $\Trop(Y)$ are the same, and are as pictured in Figure \ref{figure:cubic_cubic}.

  \begin{figure}[hbt]
\begin{center}
\includegraphics[
height=1.4in
]
{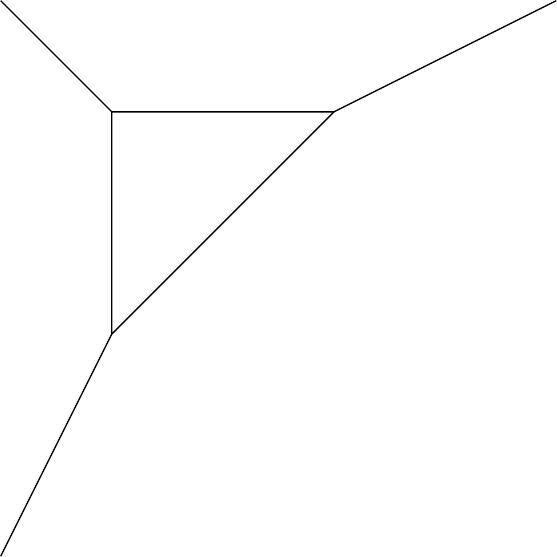}
\end{center}
\caption{ $\Trop(X)=\Trop(Y)=\Trop(X\cap Y)$, with the vertices of the triangle at $(-1,1)$, $(2,1)$, and $(-1,-2)$.}
\label{figure:cubic_cubic}
\end{figure}

The resultant of $f$ and $g$ with respect to the variable $y$ is
\begin{align*}
&t^4( c_2^2d_1^2 - 2c_1c_2d_1d_2 + c_1^2d_2^2)x^2+ t^2(c_1c_2 - c_2d_1 - c_1d_2 + d_1d_2)x^3
\\&+ t^3(-c_2c_3d_1 - c_1c_3d_2 + 2c_3d_1d_2 + 2c_1c_2d_3 -
c_2d_1d_3 - c_1d_2d_3)x^4 
\\&+t^4 (c_3^2d_1d_2 - c_2c_3d_1d_3 - c_1c_3d_2d_3 + c_1c_2d_3^2)x^5, 
\end{align*}
and the resultant of $f$ and $g$ with respect to the variable $x$ is
\begin{align*}
&t^4(c_2c_3d_1^2 - c_1c_3d_1d_2 - c_1c_2d_1d_3 + c_1^2d_2d_3)y^3 
\\&+t^3(2c_2c_3d_1 - c_1c_3d_2 - c_3d_1d_2 - c_1c_2d_3 -
c_2d_1d_3 + 2c_1d_2d_3)y^4 
\\&+t^2( c_2c_3 - c_3d_2 - c_2d_3 + d_2d_3)y^5
+t^4( c_3^2d_2^2 - 2c_2c_3d_2d_3 + c_2^2d_3^2)y^6.
\end{align*}

The stable tropical intersection consists of the three vertices of the triangle.  Let us consider possible configurations of the three intersection points that have all three intersection points lying on the triangle, rather than on the unbounded rays.  These are the configurations of zeros of rational functions   with poles precisely at the three vertices; let $h^{\trop}$ be such a function.  Label the vertices clockwise starting with $(-1,1)$ as $v_1$, $v_2$, $v_3$. Starting from $v_1$ and going clockwise, label the poles of $h^{\trop}$ as $w_1$, $w_2$, $w_3$.  Let $\delta_i$ denote the signed lattice distance between $v_i$ and $w_i$, with counterclockwise distance negative.  Then a necessary condition for the $w_i$'s to be the poles of $h^{\trop}$ is $\delta_1+\delta_2+\delta_3=0$; and in fact this condition is sufficient to guarantee the existence of such an $h^{\trop}$.  It follows that the $w_i$'s cannot be in all different or all the same line segment of triangle, as all different would have $\delta_1+\delta_2+\delta_3>0$ and all the same would have $\delta_1+\delta_2+\delta_3\neq0$.  Hence we need only show that each configuration with exactly two $w_i$'s on the same edge satisfying $\delta_1+\delta_2+\delta_3=0$ is achievable.

There are six cases to handle, since there are three choices for the edge with a pair of points and then two choices for the edge with the remaining point point.  We will focus on the case where $w_1$ and $w_2$ are on the edge connecting $v_1$ and $v_2$, and $w_3$ is on the edge connecting $v_2$ and $v_3$, as shown in Figure \ref{figure:triangle}.  Let $\delta_1=r$ and $\delta_2=-s$, where $r,s>0$, and $2-s\geq -1+r$.  It follows that $\delta_3=-(r-s)$, and that $r>s$ by the position of $w_3$.

  \begin{figure}[hbt]
\begin{center}
\includegraphics[
height=1.4in
]
{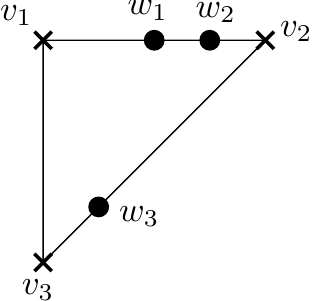}
\end{center}
\caption{The desired configuration of intersection points, where $\delta_1=r>0$, $\delta_2=-s<0$, and $\delta_3=-(r-s)<0$.}
\label{figure:triangle}
\end{figure}

To achieve the configuration specified by $r$ and $s$, set
$$c_1=3+t^r,c_2=3,c_3=1,d_1=3, d_2=d+2t^{r-s},d_3=2. $$
The valuations of the coefficients of the resultant polynomial with $x$ terms are $4+2(r-s)$ for $x^2$, $2+2r-3$ for $x^3$, $3+r-s$ for $x^4$, and $4$ for $x^5$.  It follows that the valuations of the $x$-coordinates are $2-s$, $-1+r$, and $-1-s+r$.  When coupled with rational function restrictions, this implies that the intersection points of $X$ and $Y$ tropicalize to $(-1+r,1)$, $(2-s,1)$, and $(-1-s+r, -2-s+r)$, which are indeed the points $w_1$, $w_2$, and $w_3$ we desired.

The five other cases with all three intersection points in the triangle are handled similarly, and the cases with one or more intersection point on an infinite ray are even simpler.

\end{example}

These examples provide not only a helpful check of Theorem \ref{main_theorem}, but also evidence that all possible intersection configurations can in fact be achieved.  Future work towards proving this might be of a Berkovich flavor, as in Sections \ref{section:background} and \ref{section:main_result}, or may have more to do with tropical modifications, as presented in Section \ref{section:modifications}.  Regardless of the approach, future investigations should not only look towards proving Conjecture \ref{main_conjecture}, but also towards algorithmically lifting tropical intersection configurations to curves yielding them.










\end{document}